\documentclass{aims}
\usepackage{amsmath}
  \usepackage{paralist}
  \usepackage{graphics} 
  \usepackage{epsfig} 
\usepackage{graphicx}  \usepackage{epstopdf}
 \usepackage[colorlinks=true]{hyperref}
\hypersetup{urlcolor=blue, citecolor=red}

\def\currentvolume{X}
 \def\currentissue{X}
  \def\currentyear{200X}
   \def\currentmonth{XX}
    \def\ppages{X--XX}
     \def\DOI{10.3934/xx.xx.xx.xx}

\newtheorem{theorem}{Theorem}[section]
\newtheorem{corollary}{Corollary}

\newtheorem{lemma}[theorem]{Lemma}

\theoremstyle{definition}
\newtheorem{definition}[theorem]{Definition}
\newtheorem{remark}{Remark}
\newtheorem{example}{Example}

\title[Chaos of time-varying discrete systems] 
      {Strong Li-Yorke chaos for time-varying discrete dynamical systems
with $A$-coupled-expansion}

\author[Hua Shao, Yuming Shi and Hao Zhu]{}


 \email{huashaosdu@163.com}
 \email{ymshi@sdu.edu.cn}
 \email{haozhusdu@163.com}

\thanks{This research was supported by the NNSF of Shandong Province (grantZR2011AM002).}
\thanks{* The corresponding author.}

\begin{document}
\maketitle

\centerline{\scshape Hua Shao, Yuming Shi$^{*}$ and Hao Zhu}
\medskip
{\footnotesize
 \centerline{Department of Mathematics,}
   \centerline{Shandong University}
   \centerline{ Jinan, Shandong 250100, P.~R. China}
} 

\medskip

\bigskip

 \centerline{(Communicated by the associate editor name)}

\begin{abstract}
This paper is concerned with strong Li-Yorke chaos induced by
$A$-coupled-expansion for time-varying (i.e., nonautonomous) discrete 
systems in metric spaces. Some criteria of chaos in the strong sense of
Li-Yorke are established via strict coupled-expansions for irreducible
transition matrices in bounded and closed subsets of complete metric spaces
and in compact subsets of metric spaces, respectively, where their
conditions are weaker than those in the existing literature. One example
is provided for illustration.
\end{abstract}

\section{Introduction}

In this paper, we study the following time-varying (i.e., nonautonomous) discrete system:
 $$x_{n+1}=f_n(x_n),\;\;n\geq0,\ \ \ \                                                                          \eqno(1.1)$$
where $f_n: D_n\rightarrow D_{n+1}$ is a map and $D_n$ is a set of a metric space $(X, d)$.

When $f_n=f$ and $D_n=D$ for all $n\geq0$, (1.1) is the following autonomous discrete dynamical system:
 $$x_{n+1}=f(x_n),\;\;n\geq0,\ \ \ \                                                                             \eqno(1.2) $$
where $f: D\subset X\rightarrow D$ is a map.

The autonomous discrete dynamical system (1.2) is governed by the single map $f$ while the time-varying discrete dynamical system (1.1) is generated by iteration of a sequence of maps in an order. So, it should be more difficult to study dynamical behaviors of (1.1) than those of (1.2) in general.
However, many physical, biological and economical complex systems are necessarily described by time-varying systems.
In fact, they occur more often than autonomous systems in the real world. Hence, recently, many scientists and mathematicians focused on complexity
of time-varying discrete systems ([1, 2, 4, 5, 7, 9--11, 13, 17, 26, 28--30, 32]).

It is well known that chaos is a kind of qualitative description of complexity of dynamical systems.
This concept was first introduced by Li and Yorke in 1975
when they investigated continuous interval maps in [18]. For autonomous discrete dynamical systems,
many elegant results about chaos have been obtained [3, 6, 8, 12, 15, 16, 18--20, 22--25, 27, 31, 33--35]. In particular, chaos induced by coupled-expansion has attracted a lot of attention from many scholars [3, 15, 23--25, 27, 31, 33--35], since it is easier to verify whether a system is
coupled-expanding than whether it has a transverse homoclinic orbit or regular snap-back repeller which can also induce chaos under certain conditions. Coupled-expanding maps were originally called turbulent by Block and Coppel when they studied continuous interval maps in 1992 [3]. They investigated their properties including topological entropy, period points, and chaos. In 2004, Yang and Tang extended the one-dimensional result [3, Chapter II, Proposition 15] to metric spaces [31]. In 2004, Chen and the second author of the present paper proved that a strictly coupled-expanding map,
which satisfies an expanding condition in distance on bounded closed subsets in a complete metric space or compact subsets in a metric space,
is topologically conjugate to a one-sided symbolic dynamical system [23]. In 2006, they generalized the concept of turbulence to maps in general metric spaces [25], and introduced a new term as coupled-expansion in order to avoid possible confusion with the term turbulence in fluid mechanics [24].
Later, this concept was extended to coupled-expanding maps associated with a transition matrix $A$ by the second author of the present paper with her coauthors, and several criteria of chaos induced by strict $A$-coupled-expansion were established [27, 33].

Recently, several results about chaos were obtained for
time-varying discrete systems [5, 13, 26, 28--30, 32]. Among them,
some works have been done about chaos induced by coupled-expansion
[5, 26, 28, 32]. In 2009, related concepts of chaos and coupled-expansion
were extended to time-varying discrete systems and a criterion of chaos
was established via coupled-expansion
for an irreducible transition matrix [26].
Based on this work, a criterion of chaos was given via coupled-expansion
for a special irreducible transition matrix in 2012 [32].
However, the assumptions of these existing results are stronger
than those related results for autonomous systems. So,
it is natural to ask whether these strong assumptions could be
weakened even such that the criteria of chaos are as elegant
as those for autonomous systems. In the present paper, we shall
study this problem and obtain several results which relax the
assumptions of some results in [26, 32] and extend these results
for autonomous systems to time-varying systems.

The rest of the paper is organized as follows. Section 2 presents some basic concepts and useful lemmas.
In Section 3, several criteria of Li-Yorke chaos induced by strict $A$-coupled-expansion in bounded and closed subsets in complete metric spaces
are established for system (1.1). In Section 4, other several criteria of Li-Yorke chaos induced by strict $A$-coupled-expansion
in compact subsets in metric spaces are established for system (1.1). Finally, an example is given to illustrate the theoretical results obtained in Section 5.

\section{Preliminaries}

In this section, some basic concepts for time-varying discrete dynamical systems are introduced, including chaos in the (strong) sense of Li-Yorke and coupled-expansion for a transition matrix. In addition, some basic concepts related to symbolic dynamical systems and some useful lemmas are also presented.


\begin{definition} ([26, Definition 2.7]) Let $S$ be a subset of $D_{0}$, containing at least two distinct points.
Then, $S$ is called a scrambled set of system (1.1) if, for any two distinct points $x_{0}, y_{0}\in S$, the two corresponding orbits satisfy
$${\rm (i)} \liminf\limits_{n\rightarrow\infty}^{}d(x_{n}, y_{n})=0, \;\;{\rm (ii)} \limsup\limits_{n\rightarrow\infty}^{}\emph{}d(x_{n}, y_{n})>0.$$
Further, $S$ is called a $\delta$-scrambled set for some positive constant $\delta$ if, for any two distinct points $x_{0}, y_{0}\in S$, (i) holds and,
instead of (ii), the following holds:
\begin{itemize}
\item[{\rm (iii)}] $\limsup\limits_{n\rightarrow\infty}^{}\emph{}d(x_{n}, y_{n})>\delta.$
\end{itemize}
\end{definition}

\begin{definition} ([26, Definitions 2.8 and 2.9]) If system (1.1) has an uncountable scrambled set,
then it is said to be chaotic in the sense of Li-Yorke. Furthermore, system (1.1) is said to be chaotic in the strong sense of Li-Yorke
if it has an uncountable $\delta$-scrambled set $S$ such that all the orbits starting from the points in $S$ are bounded.
\end{definition}

\begin{lemma} ({\rm[26, Theorem 2.3]}) Assume that there exists a positive integer $n_0$ such that the system
$$x_{n+1}=f_{n_0+n}(x_n),\;n\geq0,\ \ \ \ \eqno(2.1)$$
is chaotic in the strong sense of Li-Yorke; that is, system {\rm(2.1)} has an uncountable $\delta$-scrambled set $S\subset D_{n_0}$ such that
all the orbits of the system, starting from $S$, are bounded. If $(f_{n_{0}-1}\circ f_{n_{0}-2}\circ\cdots\circ f_{0})(D_0)\supset S$,
then the system {\rm(1.1)} is chaotic in the strong sense of Li-Yorke.
\end{lemma}

\begin{lemma} ({\rm[23, Lemma 2.7]}) Let $(X,d)$ be a complete metric space, and $\{A_n\}$ be a sequence of bounded and closed subsets of $X$
such that the intersection of any finitely many subsets is nonempty. If the diameter $d(A_n)\rightarrow0$ as $n\rightarrow\infty$,
then $\bigcap\limits_{n\geq1}A_n\neq\emptyset$. Furthermore, $\bigcap\limits_{n\geq1}A_n$ contains only one point.
\end{lemma}

The following definitions and lemmas are related to matrices and symbolic dynamical systems. For convenience, we first briefly recall
some properties of one-sided symbolic dynamical system as follows, which will also be used in the sequent sections. For more details, see [35].\medskip

Let $S=\{1, 2, \cdots, N\},\; N\geq2$. The one-sided sequence space
$\Sigma_{N}^{+}:=\{\alpha=(a_{0}, a_{1}, a_{2},\cdots):\; a_{i}\in S, \;i\geq 0\}$
is a metric space with the distance
$\rho(\alpha, \beta)=\sum_{i=0}^{\infty}d(a_{i}, b_{i})/2^{i}$,
where $\alpha=(a_{0}, a_{1}, a_{2},\cdots)$, $\beta=(b_{0}, b_{1}, b_{2},\cdots)\in \Sigma_{N}^{+}$,
$d(a_{i}, b_{i})=1$ if $a_{i}\neq b_{i}$, and $d(a_{i}, b_{i})=0$ if $a_{i}=b_{i}$, $i\geq 0$. It is a complete metric space and a Cantor set.
Define the shift map $\sigma: \Sigma_{N}^{+}\rightarrow \Sigma_{N}^{+}$ by
$\sigma(\alpha):=(a_{1}, a_{2},\cdots)$, where $\alpha=(a_{0}, a_{1}, a_{2},\cdots)$. This map is continuous and $(\Sigma_{N}^{+}, \sigma)$ is called the one-sided symbolic dynamical system on $N$ symbols.\medskip

A matrix $A=(a_{ij})_{N\times N}\;(N\geq2)$ is said to be a transition matrix if $a_{ij}=0$ or 1 for all $i,j$; $\sum\limits_{j=1}^{N}a_{ij}\geq 1$ for all $i$; and $\sum\limits_{i=1}^{N}a_{ij}\geq 1$ for all $j$, $1\leq i,j\leq N$ [14]. For a given transition matrix $A=(a_{ij})_{N\times N}$, denote
$\Sigma_{N}^{+}(A):=\{s=(s_{0}, s_{1}, \cdots, s_{n}, \cdots )$ $\in\Sigma_{N}^{+}: a_{s_{j}s_{j+1}}=1, \;j\geq0\}$.
Then $\Sigma_{N}^{+}(A)$ is a compact subset of $\Sigma_{N}^{+}$ and invariant under $\sigma$.
The map $\sigma_{A}:=\sigma\mid_{\Sigma_{N}^{+}(A)}:\Sigma_{N}^{+}(A)\rightarrow \Sigma_{N}^{+}(A)$ is said to be the subshift of finite type for matrix $A$.\medskip

A non-negative  matrix $A=(a_{ij})_{N\times N}\;(N\geq2)$ is irreducible if and only if for each pair $(i,j),\;1\leq i,j\leq N$, there exists a positive integer $k$ such that $a_{ij}^{(k)}>0$, where $a_{ij}^{(k)}$ denotes the $(i,j)$ entry of matrix $A^{k}$ [21]. Given an irreducible transition matrix $A=(a_{ij})_{N\times N}$, a finite
sequence $w=(s_1,\cdots,s_k)$ is called an allowable word of length $k$ for $A$ if $a_{s_{i}s_{i+1}}=1, \;1\leq i\leq k-1$, where $s_{1},s_{2},\cdots,s_{k}\in S$ [22]. The length $k$ of $w$ is often denoted by $|w|$.\medskip

Now, we introduce the following two properties of transition matrices:\medskip

\begin{lemma} ([22]) Let $A=(a_{ij})_{N\times N}$ be a transition matrix. Then, for
any $k\geq1$ and any $1\leq i, j \leq N$, there are exactly $a^{(k)}_{ij}$
allowable words of length $k +1$ for $A$, starting at $i$ and ending at $j$, in the form of $(i, s_{1}, s_{2},\cdots, s_{k-1}, j)$.
\end{lemma}

\begin{lemma} ({\rm[33, Theorem 2.2]}) Let $A=(a_{ij})_{N\times N}\;(N\geq2)$ be an irreducible transition matrix
with $\sum\limits_{j=1}^{N}a_{i_{0}j}\geq2$ for some $1\leq i_{0}\leq N$. Then
\begin{itemize}
\item [{\rm (a)}] for any $1\leq i,j\leq N$ and any positive integer $l$, there exists at least one allowable word
$w=(i,\cdots, j)$ for $A$ such that $|w|>l$;
\item [{\rm (b)}] for any given allowable word $w=(b_{1}, b_{2},\cdots, b_{k})$ for $A$, if $k>2N\times(N^{2}-2N+2)$, then there exists another different allowable word  $w'=(c_{1}, c_{2},\cdots, c_{k})$ for $A$ with $c_{1}=b_{1}$ and $c_{k}=b_{k}$.
\end{itemize}
\end{lemma}

\begin{definition} ([26, Definition 2.10]) Assume that $\bigcap\limits_{n=0}^{\infty}D_{n}$ is not empty,
and $A=(a_{ij})_{N\times N}$ is a transition matrix ($N\geq2$). If there exist $N$ subsets $V_{i}$ $(1\leq i\leq N)$ of $D$ with
$V_{i}\cap V_{j}=\partial_{D} V_{i}\cap\partial_{D} V_{j}$ for each pair $(i, j) ,\; 1\leq i\neq j\leq N$, such that
$$f_{n}(V_{i})\supset {\bigcup_{a_{ij}=1}}V_{j},\;\; 1\leq i\leq N,\;\; n\geq 0,$$
where $\partial_{D} V_{i}$ is the relative boundary of $V_{i}$ with respect to $D$, then system (1.1) (or the sequence of maps $\{f_{n}\}_{n=0}^{\infty}$)
is said to be coupled-expanding for matrix $A$ (or $A$-coupled-expanding) in $V_{i}, 1\leq i\leq N$.
Further, system (1.1) (or the sequence of maps $\{f_{n}\}_{n=0}^{\infty}$) is said to be strictly coupled-expanding for matrix $A$
(or strictly $A$-coupled-expanding) in $V_{i}$, $1\leq i\leq N$, if $d(V_{i}, V_{j})>0$ for all $1\leq i\neq j\leq N$.
In the special case that $a_{ij}=1$ for all $1\leq i,j\leq N$,
it is briefly said to be coupled-expanding or strictly coupled-expanding in $V_{i}, 1\leq i\leq N$.
\end{definition}

Coupled-expansion is a powerful tool to determine whether a system is chaotic
since it can be done directly by properties of the maps governing the system.
The following two results are existing criteria of chaos induced by strictly $A$-coupled-expansion for time-varying systems and autonomous systems, respectively.

\begin{lemma} ({\rm[26, Theorem 4.1]}) Let $(X,d)$ be a complete metric space. Assume that there exists an integer $n_0\geq0$ such that
$\bigcap\limits_{n=n_0}^{\infty}D_{n}\neq\emptyset$, and assume that there exist $N\;(N\geq2)$ bounded and closed subsets $V_{j}$
of $\bigcap\limits_{n=n_0}^{\infty}D_{n}$,\;$1\leq j\leq N$, with $d(V_{i},V_{j})>0,\;1\leq i\neq j\leq N$, such that
$f_{n}$ are continuous in $E:=\bigcup\limits_{j=1}^{N}V_{j}$ for all $n\geq n_0$. If
\begin{itemize}
\item [{\rm (i)}] $(f_{n_{0}-1}\circ f_{n_{0}-2}\circ\cdots\circ f_{0})(D_0)\supset\bigcup\limits_{j=1}^{N}V_{j}$;
\item [{\rm (ii)}] there exists $N\times N$ irreducible transition matrix $A$ with $\sum\limits_{j=1}^{N}(A)_{i_{0}j}\geq2$ for some $i_{0}$ such that
system {\rm(1.1) ($n\geq n_0$)} is strictly $A$-coupled-expanding in $V_{j}$,\;$1\leq j\leq N$;
\item [{\rm (iii)}] there exist constants $\mu\ge\lambda>1$ such that for all $n\geq n_0$,
$$\lambda d(x,y)\leq d(f_{n}(x),f_{n}(y))\leq\mu d(x,y),\;\forall\;x,y\in V_{j},\; 1\leq j\leq N.             \eqno(2.2)$$
\end{itemize}
Then system {\rm(1.1)} is chaotic in the strong sense of Li-Yorke.
\end{lemma}

\begin{lemma} ({\rm[33, Theorem 3.2]}) Let $(X, d)$ be a complete metric space, and $V_1,\cdots,V_N$ be disjoint bounded and closed nonempty sets of $X$
with $d(V_{i},V_{j})>0,\;1\leq i\neq j\leq N$. Let $A =(a_{ij})_{N\times N}$ be an irreducible transition matrix with $\sum\limits_{j=1}^{N}a_{i_{0}j}\geq2$
for some $1\leq i_0\leq N$. Suppose that a map $f : D :=\bigcup\limits_{i=1}^{N}V_{i}\rightarrow X$ is continuous and satisfies the following
\begin{itemize}
\item [{\rm (i)}] the map $f$ is strictly $A$-coupled-expanding in $V_1,\cdots,V_N$;
\item [{\rm (ii)}] there exist positive constants $\lambda>1$, $\mu$ and an integer $1\leq j_{0}\leq N$ such that
$$d(f(x),f(y))\geq \lambda d(x,y),\;\forall\;x,y\in V_{j_{0}},                                                \eqno(2.3)$$
$$d(f(x),f(y))\geq \mu d(x,y),\;\forall\;x,y\in V_{j},\;j\neq j _{0};                                         \eqno(2.4)$$      and
$$\lambda\mu^{k_{0}-1}>1;                                                                                     \eqno(2.5)$$
where $k_{0}$ is the minimal positive integer such that
$$a_{j_{0}j_{0}}^{(k_{0})}>0.                                                                                 \eqno(2.6)$$
\end{itemize}
Then
\begin{itemize}
\item [{\rm (1)}] the map $f$ has an uncountable and bounded scrambled set $S \subset D$, so $f$ is chaotic in the sense of Li-Yorke ;
\item [{\rm (2)}] there exists an uncountable, perfect, bounded, and closed set $E\subset D$ such that $S\subset E$, $f(E)= E$, and $f$ is chaotic on $E$ in the sense of Devaney ;
\item [{\rm (3)}] there exist positive integers $N_1$ and $N_2$ such that for any integer $l\geq N_1$, $f$ has a $k$-periodic point in $E$, with
$k = lk_0 + N_2$ as the minimal period.
\end{itemize}
\end{lemma}

Since $A$ is an irreducible transition matrix, the existence of $k_{0}$ in (2.6) is obvious.

Comparing these two Lemmas, it is not difficult to find that assumption (iii) in Lemma 2.8 for time-varying systems is much stronger than assumption (ii) in Lemma 2.9 for autonomous systems since the control of upper bounds is required in (2.2) and the control of lower bounds in (2.2) is also stronger than that in (2.3)-(2.6). In the present paper, we try to relax this assumption in the following two sections.

\section{ Li-Yorke chaos induced by strict $A$-coupled-expansion in bounded and closed subsets in complete metric spaces}

\begin{theorem}
       Let all the assumptions in Lemma 2.8 hold except that assumption $\rm{(iii)}$ is replaced by
\begin{itemize}
\item [{\rm (iii$_{a}$)}] there exist positive constants $\lambda>1,\;\mu,$ and an integer $1\leq j_{0}\leq N$, such that for all $n\geq n_0$,
$$d(f_{n}(x),f_{n}(y))\geq \lambda d(x,y),\;\forall\;x,y\in V_{j_{0}},                              \ \ \ \ \ \eqno(3.1)$$
$$d(f_{n}(x),f_{n}(y))\geq \mu d(x,y),\;\forall\;x,y\in V_{j},\;j\neq j _{0};                                 \eqno(3.2)$$
and (2.5) and (2.6) hold.
\end{itemize}
Then system {\rm(1.1)} is chaotic in the strong sense of Li-Yorke.
\end{theorem}

\begin{proof}
    By Lemma 2.3, it suffices to show that system (2.1)
is chaotic in the strong sense of Li-Yorke. The rest of the proof is divided into three steps.

{\bf Step 1.} We claim that there exists $\alpha=(a_{0},a_{1},a_{2},\cdots)\in \sum_{N}^{+}(A)$ such that $d(V_{\alpha}^{l_{i},m+n_0})$ uniformly converges to 0 with respect to $m\geq0$ as $i\rightarrow\infty$, where $l_i=ik_0$, $k_0$ is specified by (2.5) and (2.6), and
$$V_{\alpha}^{n,i}:=\bigcap\limits_{k=0}^{n}f_{i}^{-k}(V_{a_{k}}),\;\; i\geq0,\eqno(3.3)$$
where $f_{i}^{k}=f_{i+k-1}\circ \cdots \circ f_{i+1}\circ f_{i}$,\;\; $f_{i}^{-k}=(f_{i}^{k})^{-1}$.

By (2.6) and Lemma 2.5 there exists an allowable word $(j_{0},b_{1},\cdots,b_{k_{0}-1},j_{0})$ with length $k_{0}+1$ for matrix $A$.
Set
$$\alpha=(a_{0},a_{1},a_{2},\cdots):=(j_{0},b_{1},\cdots,b_{k_{0}-1},j_{0},b_{1},\cdots,b_{k_{0}-1},j_{0},\cdots).\eqno(3.4)$$
Note that $a_{l_i}=j_{0}$ for $i\geq0$. Evidently, $\alpha\in \sum_{N}^{+}(A)$.
It follows from assumption (iii) that for any $i\geq 1$ and any $x,y\in V_{\alpha}^{l_{i},m+n_0}$,
$$d(f_{m+n_0}^{l_{i}}(x),f_{m+n_0}^{l_{i}}(y))\geq \lambda\mu^{k_{0}-1}d(f_{m+n_0}^{l_{i-1}}(x),f_{m+n_0}^{l_{i-1}}(y))\geq (\lambda\mu^{k_{0}-1})^{i}d(x,y),$$
which implies that
$$d(x,y)\leq (\lambda\mu^{k_{0}-1})^{-i}d(f_{m+n_0}^{l_{i}}(x),f_{m+n_0}^{l_{i}}(y))\leq (\lambda\mu^{k_{0}-1})^{-i}d(V_{a_{l_{i}}})=
(\lambda\mu^{k_{0}-1})^{-i}d(V_{j_{0}}).$$
Thus
$$d(V_{\alpha}^{l_{i},m+n_0})\leq (\lambda\mu^{k_{0}-1})^{-i}d(V_{j_{0}}), \;i\geq 1.$$
This, together with (2.5), yields that $d(V_{\alpha}^{l_{i},m+n_0})$ uniformly converges to 0 with respect to $m\geq0$ as $i\rightarrow\infty$.

{\bf Step 2.} Construct the scrambled set $S$.

Since $A$ is a transition matrix, there exist $1 \leq t_{0},m_{0}\leq N$ such that
$a_{t_{0}j_{0}}=a_{j_{0}m_{0}}=1$. Moreover, $A$ is irreducible and satisfies that $\sum\limits_{j=1}^{N}a_{i_{0}j}\geq2$ for some $i_{0}$. So, by Lemma 2.6 there exist three allowable words for matrix $A$:
$$w_{0}=(c_{1},c_{2},\cdots,c_{m_{1}}),\; w_{1}=(d_{1},d_{2},\cdots,d_{m_{2}}),\; w_{2}=(e_{1},e_{2},\cdots,e_{m_{2}}),\ \ \ \ \ \eqno(3.5)$$
where
$$w_{1}\neq w_{2},\; c_{1}=m_{0},\; d_{1}=e_{1}=j_{0}, \;c_{m_{1}}=d_{m_{2}}=e_{m_{2}}=t_{0}.\ \ \ \ \ \eqno(3.6)$$
Set
$$\Omega:=\{\beta=(w_{0},B_{1},B_{2},\cdots):B_{i}= w_{1}\;{\rm or} \; w_{2}, \;\;i\geq 1\}.$$
Evidently, $\Omega$ is an uncountable subset of $\sum_{N}^{+}(A)$.

By Step 1, we have shown that $d(V_{\alpha}^{l_{i},m+n_0})$ uniformly converges to 0 with respect to $m\geq0$ as $i\rightarrow\infty$.
Thus, there exists a strictly increasing subsequence $\{l_{i_{j}}\}_{j=1}^{\infty}$ of $\{l_{i}\}_{i=1}^{\infty}$ such that for any $j\geq1$,
$$d(V_{\alpha}^{l_{i_{j}},h_{j-1}+n_0})\leq \mu_{0}^{h_{j-1}}2^{-j},\ \ \ \ \ \eqno(3.7)$$
where $\mu_{0}=\min\{\lambda,\mu\}$, and
$$h_{0}=0,\;\; h_{j}=\sum\limits_{t=1}^{j}l_{i_{t}}+jm_{1}+j(j-1)m_{2}/2+j,\;\; j\geq 1.$$
For any $\beta=(w_{0},B_{1},B_{2},\cdots)\in \Omega$, set
$$\hat{\beta}:=(a_{0},\cdots,a_{l_{i_{1}}},w_{0},a_{0},\cdots,a_{l_{i_{2}}},w_{0},B_{1},a_{0},\cdots,a_{l_{i_{3}}},w_{0},B_{1},B_{2},\cdots),\eqno(3.8)$$
where $a_j\;(j\geq0)$ is specified by (3.4). Then it follows from (3.5) and (3.6) that $\hat{\beta}\in \sum_{N}^{+}(A)$.

Next, it is to show that $\lim\limits_{n\rightarrow\infty}d(V_{\hat{\beta}}^{n,n_0})=0$.
For each $j\geq 1$, it follows from (3.8) that if $x\in V_{\hat{\beta}}^{h_{j}-1,n_0}$, then
$$f_{n_0}^{h_{j-1}+i}(x)\in V_{a_{i}},\;\;0\leq i\leq l_{i_{j}},$$
which implies that
$$f_{n_0}^{h_{j-1}}(x)\in V_{\alpha}^{l_{i_{j}},h_{j-1}+n_0},\;\; j\geq 1.$$
So, for any $x,y\in V_{\hat{\beta}}^{h_{j}-1,n_0}$, one has that
$$f_{n_0}^{h_{j-1}}(x),f_{n_0}^{h_{j-1}}(y)\in V_{\alpha}^{l_{i_{j}},h_{j-1}+n_0},\;\; j\geq 1.\ \ \ \ \ \eqno(3.9)$$
Thus
$$d(f_{n_0}^{h_{j-1}}(x),f_{n_0}^{h_{j-1}}(y))\leq d(V_{\alpha}^{l_{i_{j}},h_{j-1}+n_0}),\;\; j\geq 1.\ \ \ \ \ \eqno(3.10)$$
Further, by assumption (iii), (3.7), and (3.10) we get that for any $x,y\in V_{\hat{\beta}}^{h_{j}-1,n_0}$,
$$d(x,y)\leq \mu_{0}^{-h_{j-1}}d(f_{n_0}^{h_{j-1}}(x),f_{n_0}^{h_{j-1}}(y))\leq \mu_{0}^{-h_{j-1}}d(V_{\alpha}^{l_{i_{j}},h_{j-1}+n_0})
\leq 2^{-j},$$
which results in
$$d(V_{\hat{\beta}}^{h_{j}-1,n_0})\leq 2^{-j},\; j\geq 1.$$
Hence,
$$\lim_{j\rightarrow\infty}d(V_{\hat{\beta}}^{h_{j}-1,n_0})=0.$$
Noting the fact that $\{d(V_{\hat{\beta}}^{n,n_0})\}_{n=0}^{\infty}$ is a non-increasing sequence,
one obtains that
$$\lim_{n\rightarrow\infty}d(V_{\hat{\beta}}^{n,n_0})=0.\ \ \ \ \ \eqno(3.11)$$
By assumption (ii) and the continuity of $f_n$, it can be verified that $V_{\hat{\beta}}^{n,n_0}$ are  nonempty and
closed subsets of $V_{a_0}$ for all $n\geq0$, and they form a nested family of sets; that is,
$V_{\hat{\beta}}^{0,n_0}\supset V_{\hat{\beta}}^{1,n_0}\supset V_{\hat{\beta}}^{2,n_0}\supset\cdots$.
This, together with (3.11), yields that $\bigcap\limits_{n\geq 0}V_{\hat{\beta}}^{n,n_0}$ is a singleton set by Lemma 2.4.
Denote
$$\{x_{\hat{\beta}}\}:=\bigcap_{n\geq 0}V_{\hat{\beta}}^{n,n_0},\;\; S:=\{x_{\hat{\beta}}:\beta \in \Omega\}.$$
It is evident that $\hat{\beta_{1}}\neq \hat{\beta_{2}}$ if and only if $\beta_{1}\neq \beta_{2}$.
So, $x_{\hat{\beta_{1}}}\neq x_{\hat{\beta_{2}}}$ for any $\beta_{1},\beta_{2}\in\Omega$ with $\beta_{1}\neq \beta_{2}$. This shows that $S$ is an uncountable set.

{\bf Step 3.} Show that $S$ is a $\delta$-scrambled set of system (2.1), where $\delta$:= min$\{d(V_{i},V_{j}),$
$1\leq i \neq j \leq N\}>0$.

Fix any given $x_{\hat{\beta_{1}}},x_{\hat{\beta_{2}}} \in S$ with $x_{\hat{\beta_{1}}}\neq x_{\hat{\beta_{2}}}$.
Then there exist $\beta_{1}\neq\beta_{2}\in\Omega$ such that $\hat{\beta_{1}}$ and $\hat{\beta_{2}}$ are determined by $\beta_{1}$ and $\beta_{2}$, respectively, by using (3.8). So, it follows from (3.8) that there exist two different symbols $s_{1},s_{2}\in \{1,\cdots,N\}$ and an increasing sequence $\{n_{j}\}_{j=1}^{\infty}$ of positive integers such that $f_{n_0}^{n_{j}}(x_{\hat{\beta_{1}}})\in V_{s_{1}}$ and $f_{n_0}^{n_{j}}(x_{\hat{\beta_{2}}})\in V_{s_{2}}$,
which implies that
$$d(f_{n_0}^{n_{j}}(x_{\hat{\beta_{1}}}),f_{n_0}^{n_{j}}(x_{\hat{\beta_{2}}}))\geq d(V_{s_{1}},V_{s_{2}})\geq \delta, \;\;j\geq 1.$$
Thus, one has that
$$\limsup_{n\rightarrow\infty}d(f_{n_0}^{n}(x_{\hat{\beta_{1}}}),f_{n_0}^{n}(x_{\hat{\beta_{2}}}))\geq \delta.\ \ \ \ \ \eqno(3.12)$$
On the other hand, by (3.9) one has that
$$f_{n_0}^{h_{j-1}}(x_{\hat{\beta_{1}}}),f_{n_0}^{h_{j-1}}(x_{\hat{\beta_{2}}})\in V_{\alpha}^{l_{i_{j}},h_{j-1}+n_0}, \;\;j\geq 1.$$
Then
$$d(f_{n_0}^{h_{j-1}}(x_{\hat{\beta_{1}}}),f_{n_0}^{h_{j-1}}(x_{\hat{\beta_{2}}}))\leq d(V_{\alpha}^{l_{i_{j}},h_{j-1}+n_0}), \;\;j\geq 1.$$
Since $d(V_{\alpha}^{l_{i},m+n_0})$ uniformly converges to 0 with respect to $m$ as $i\rightarrow\infty$, we get that
$$\liminf_{n\rightarrow\infty}d(f_{n_0}^{n}(x_{\hat{\beta_{1}}}),f_{n_0}^{n}(x_{\hat{\beta_{2}}}))=0.\ \ \ \ \ \eqno(3.13)$$
Hence, $S$ is an uncountable $\delta$-scrambled set of system (2.1) by (3.12) and (3.13).

Moreover, for any $x_{n_0}\in S$, its orbit  $\{x_{n}\}_{n=n_0}^{\infty}$ under system (2.1) satisfies that $\{x_{n}\}_{n=n_0}^{\infty}\subset\bigcup\limits_{i=1}^{N}V_{i}$.
Then, $\{x_{n}\}_{n=n_0}^{\infty}$ is bounded since $\bigcup\limits_{i=1}^{N}V_{i}$ is bounded. Hence, all the orbits starting from the points in $S$ are bounded.
Therefore, system (2.1) is chaotic in the strong sense of Li-Yorke.

Consequently, system (1.1) is chaotic in the strong sense of Li-Yorke by assumption (i) and Lemma 2.3.
This completes the proof.
\end{proof}

\begin{remark} The method used in the proof of Theorem 3.1 is motivated by that of Lemma 2.9, in which
a strictly $A$-coupled-expanding map was studied.
\end{remark}

If $\min\{\lambda,\mu\}>1$ in assumption (iii) of Theorem 3.1, one has the following better result, which is a direct consequence of Theorem 3.1.

\begin{corollary} Let all the assumptions in Lemma 2.8 hold except that assumption $\rm{(iii)}$ is replaced by
\begin{itemize}
\item [${\rm (iii_{b})}$] there exists a positive constant $\lambda>1$ such that for all $n\geq n_{0}$,
$$d(f_{n}(x),f_{n}(y))\geq \lambda d(x,y),\;\forall\;x,y\in V_{j},\;\;1\leq j\leq N.$$
\end{itemize}
Then system {\rm(1.1)} is chaotic in the strong sense of Li-Yorke.
\end{corollary}

\begin{remark} Both Theorem 3.1 and Corollary 1 relaxes the assumptions of Lemma 2.8. More exactly, the control of the upper bounds in (2.2) is removed. Similarly, Theorem 3.1 also relaxes assumption (iii) of Theorem 3.1 in [32], in which time-varying coupled-expansion for a special irreducible transition matrix was discussed. In addition, Theorem 3.1 weakens the control of the lower bounds in (2.2) and generalizes the result about Li-Yorke chaos
in Lemma 2.9 for autonomous systems to time-varying systems; Corollary 1 extends the result about Li-Yorke chaos in Theorem 5.5 in [27] for autonomous systems to time-varying systems.
\end{remark}

The following result can be directly derived from Theorem 3.1 in the case that $k_0=1$ in (2.6).

\begin{corollary} Let all the assumptions in Lemma 2.8 hold except that assumption $\rm{(iii)}$ is replaced by
\begin{itemize}
\item [${\rm (iii_{c})}$] there exist an integer $j_0,\;1\leq j_0\leq N$, and positive constants $\lambda>1$ and $\mu$ such that $a_{j_0j_0}=1$ and
for all $n\geq n_{0}$,
$$d(f_{n}(x),f_{n}(y))\geq \lambda d(x,y),\;\forall\;x,y\in V_{j_{0}},$$
$$d(f_{n}(x),f_{n}(y))\geq \mu d(x,y),\;\forall\;x,y\in V_{j},\;j\neq j _{0}.$$
\end{itemize}
Then system {\rm(1.1)} is chaotic in the strong sense of Li-Yorke.
\end{corollary}

\begin{remark} Corollary 2 extends the result on Li-Yorke chaos in Corollary 3.1 in [33] for autonomous systems to time-varying systems.\end{remark}

In [28], the second author of the present paper was motivated by the fact that the concepts of chaos are characterized by properties of orbits at some times and then introduced a new type of induced system of a time-varying discrete system as follows. Let $\{k_{n}\}_{n=1}^{\infty}$ be an increasing sequence of positive integers with $k_{n}\rightarrow\infty$ as $n\rightarrow\infty$. The following system:
$$\hat{x}_{n+1}=\hat{f_{n}}(\hat{x}_{n}),\;n\geq0,\ \ \ \ \eqno(3.14)$$
is called the induced system by system (1.1) through (or with respect to) $\{k_{n}\}_{n=1}^{\infty}$, where
$$\hat{f_{0}}:=f_{k_{1}-1}\circ f_{k_{1}-2}\circ\cdots\circ f_{0},\;\hat{f_{n}}:=f_{k_{n+1}-1}\circ f_{k_{n+1}-2}\circ\cdots\circ f_{k_{n}},\;n\geq1,$$
while $\hat{f}_{n}:\hat{D}_{n}\rightarrow\hat{D}_{n+1}$ for $n\geq0$; and $\hat{D_{0}}:=D_{0}, \hat{D}_{n}:=D_{k_{n}}$ for $n\geq1$.

Let $\{x_{n}\}_{n=0}^{\infty}$ be the orbit of system (1.1) starting from $x_{0}$ and  $\{\hat{x}_{n}\}_{n=0}^{\infty}$ be the orbit of the
induced system (3.14) starting from $\hat{x}_{0}:=x_{0}$. Then $\hat{x}_{n}=x_{k_n}, n\geq1$. So the orbit $\{\hat{x}_{n}\}_{n=0}^{\infty}$ of the induced system (3.14) is a part of the orbit $\{x_{n}\}_{n=0}^{\infty}$ of system (1.1) starting from the same initial point $x_0$. It is also said to be a suborbit of the orbit $\{x_{n}\}_{n=0}^{\infty}$ of system (1.1).

\begin{lemma} ({\rm[28, Theorem 3.1]}) For the induced system (3.14), one has that
\begin{itemize}
\item [{\rm (1)}] If system {\rm(3.14)} is chaotic
in the sense of Li-Yorke through some $\{k_{n}\}_{n=1}^{\infty}$, so is system {\rm(1.1)}.
\item [{\rm (2)}] Assume that $D_n$, $n\geq0$; are in the same metric space $(X; d)$ and $\{D_n\}_{n=m}^{\infty}$ is uniformly bounded for some integer
$m\geq0$. If system {\rm(3.14)} is chaotic in the strong sense of Li-Yorke through some $\{k_{n}\}_{n=1}^{\infty}$, so is system {\rm(1.1)}.
\end{itemize}
\end{lemma}

We get the following result applying Lemma 3.2 to Theorem 3.1:

\begin{theorem} Assume that there exists an
increasing subsequence $\{k_{n}\}_{n=1}^{\infty}$ of positive integers such that all the assumptions in Theorem 3.1 hold for system (3.14).
Then system {\rm(1.1)} is chaotic in the sense of Li-Yorke. Further,
if $\{D_n\}_{n=m}^{\infty}$ or $\{(f_{k_{n}+i}\circ f_{k_{n}+i-1}\circ\cdots\circ f_{k_{n}})(E):0\leq i\leq k_{n+1}-k_{n}-1,n\geq m\}$ is uniformly bounded for some integer $m\geq0$, where $E$ is specified by Lemma 2.8, then system {\rm(1.1)} is chaotic in the strong sense of Li-Yorke.
\end{theorem}

\begin{proof} By Theorem 3.1, the induced system (3.14) is chaotic in the strong sense of Li-Yorke. Thus, system (1.1) is chaotic in the sense of Li-Yorke by (i) of Lemma 3.2.

In the first case that $\{D_{n}\}_{n=m}^{\infty}$ is uniformly bounded for some integer $m\geq0$, it is evident that system (1.1) is chaotic
in the strong sense of Li-Yorke by (ii) of Lemma 3.2.

In the second case that $\{(f_{k_{n}+i}\circ f_{k_{n}+i-1}\circ\cdots\circ f_{k_{n}})(E) : 0 \leq i \leq k_{n+1}-k_{n}-1; n \geq m\}$ is uniformly
bounded for some integer $m\geq0$, let $S$ be the corresponding scrambled set of system (3.14) $(n\geq n_0)$. It
follows from the proof of Theorem 3.1 that
$$(\hat{f}_{n}\circ\hat{f}_{n-1}\circ\cdots\circ\hat{f}_{n_0})(S)\subset E;\;\;n\geq n_0,$$
which implies that for each initial point $\hat{x}_{n_0}\in S$, every point $\hat{x}_{n}$ in the corresponding orbit of system (3.14)
$(n\geq n_0)$ lies in $E$. So the orbit $\{\hat{x}_{n}\}$ is bounded since $V_j$,\;$1\leq j\leq N$, are bounded. Hence, the orbit of system (1.1) $(n\geq k_{n_0})$
starting from $x_{k_{n_0}}=\hat{x}_{n_0}$ is bounded by the assumption. Therefore, system (1.1) is chaotic in
the strong sense of Li-Yorke. This completes the proof.
\end{proof}

\begin{remark} Similarly, Theorem 3.3 relaxes assumption (iii) of Theorem 4.1 in [28]. By taking $k_n=n$ for $n\geq1$,
Theorem 3.3 is the same as Theorem 3.1.
\end{remark}

\section{Li-Yorke chaos induced by strict $A$-coupled-expansion in compact subsets in metric space}

In this section, we shall establish several criteria of chaos in the strong sense of Li-Yorke, induced
by strict coupled-expansion for transition matrix in compact subsets in metric spaces, for system (1.1).

\begin{theorem} Let all the assumptions in Theorem 3.1 hold except that $(X,d)$ is a metric space and $V_{j},\;1\leq j\leq N$,
are disjoint nonempty compact subsets. Then system {\rm(1.1)} is chaotic in the strong sense of Li-Yorke.
\end{theorem}

\begin{proof} Since the proof is completely similar to that of Theorem 3.1, its details are omitted.
\end{proof}

If $\min\{\lambda,\mu\}>1$ in assumption (iii) of Theorem 4.1, one has the following result, which is a direct consequence of Theorem 4.1.

\begin{corollary} Let all the assumptions in Corollary 1 hold except that $(X,d)$ is a metric space and $V_{j},\;1\leq j\leq N$,
are disjoint nonempty compact subsets. Then system {\rm(1.1)} is chaotic in the strong sense of Li-Yorke.
\end{corollary}

\begin{remark} Corollary 3 extends the result on Li-Yorke chaos of Theorem 5.2 in [27] for autonomous systems to time-varying systems.\end{remark}

If $k_{0}=1$ in assumption (iii) of Theorem 4.1,  we get the following better result that can strongly relax assumption (iii) in Theorem 4.1
because it only requires that $f_n$ are expanding in distance in one subset for all $n\geq n_{0}$.

\begin{theorem} Let all the assumptions in Theorem 4.1 hold except that assumption $\rm{(iii)}$ is replaced by
\begin{itemize}
\item [{\rm (iii$_{d}$)}] there exist an integer $j_0,\;1\leq j_0\leq N$, and positive constants $\lambda>1$ such that $a_{j_0j_0}=1$ and
for all $n\geq n_{0}$,
$$d(f_{n}(x),f_{n}(y))\geq \lambda d(x,y),\;\forall\;x,y\in V_{j_{0}}.$$
\end{itemize}
Then system {\rm(1.1)} is chaotic in the strong sense of Li-Yorke.
\end{theorem}

Although the idea in the proof of Theorem 4.2 is similar to that of Theorem 3.1, we give its detailed proof for completeness.

\begin{proof} By Lemma 2.3, it suffices to show that system (2.1)
is chaotic in the strong sense of Li-Yorke. For convenience, the rest of the proof is divided into three steps.

{\bf Step 1.} We claim that there exists $\alpha=(a_{0},a_{1},a_{2},\cdots)\in \sum_{N}^{+}(A)$ such that $d(V_{\alpha}^{n,m+n_0})$ uniformly converges to 0 with respect to $m\geq0$ as $n\rightarrow\infty$, where $V_{\alpha}^{n,m+n_0}$ is defined by (3.3).
Set
$$\alpha=(a_{0},a_{1},a_{2},\cdots):=(j_{0},j_{0},j_{0},j_{0},\cdots).$$
Then $\alpha\in \sum_{N}^{+}(A)$ since $a_{j_{0}j_{0}}=1$.
By assumption (iii), one easily obtain that
$$d(f_{m+n_0}^{n}(x),f_{m+n_0}^{n}(y))\geq \lambda^{n} d(x,y), \;\;\forall x, y \in V_{\alpha}^{n,m+n_0}, \;\;n\geq0.$$
Thus
$$d(x,y)\leq\lambda^{-n}d(f_{m+n_0}^{n}(x),f_{m+n_0}^{n}(y))\leq \lambda^{-n}d(V_{j_{0}})\leq \lambda^{-n}\tilde{d},$$
where $\tilde{d}=\max\limits_{1\leq j\leq N}\{d(V_{j})\}$. Then
$$d(V_{\alpha}^{n,m+n_0})\leq \lambda^{-n}\tilde{d},\;\;n\geq0.$$
Hence,
$d(V_{\alpha}^{n,m+n_0})$ uniformly converges to 0 with respect to $m\geq0$ as $n\rightarrow\infty$.

{\bf Step 2.} Construct the scrambled set $S$.

Since $A$ is a transition matrix, there exist $1\leq l_{0},m_{0}\leq N$ such that
$a_{l_{0}j_{0}}=a_{j_{0}m_{0}}=1$. Moreover, $A$ is irreducible and satisfies that $\sum\limits_{j=1}^{N}a_{i_{0}j}\geq2$ for some $i_{0}$. So, by Lemma 2.6 there exist three allowable words for matrix $A$:
$$w_{0}=(c_{1},c_{2},\cdots,c_{l_{1}}), w_{1}=(d_{1},d_{2},\cdots,d_{l_{2}}), w_{2}=(e_{1},e_{2},\cdots,e_{l_{2}}),\ \ \ \ \ \eqno(4.1)$$
where
$$w_{1}\neq w_{2}, c_{1}=m_{0}, d_{1}=e_{1}=j_{0}, c_{l_{1}}=d_{l_{2}}=e_{l_{2}}=l_{0}.\ \ \ \ \ \eqno(4.2)$$
Set
$$\Omega:=\{\beta=(w_{0},B_{1},B_{2},\cdots):B_{i}= w_{1}\;{\rm or} \; w_{2}, \;\;i\geq 1\}.$$
Evidently, $\Omega$ is an uncountable subset of $\sum_{N}^{+}(A)$.
For any $\beta=(w_{0},B_{1},B_{2},\cdots)\in \Omega$, set
$$\hat{\beta}=(b_{0},b_{1},\cdots):=(j_{0},w_{0},j_{0},j_{0},w_{0},B_{1},j_{0},j_{0},j_{0},w_{0},B_{1},B_{2},j_{0},\cdots).\ \ \ \ \ \eqno(4.3) $$
Then it follows from (4.1) and (4.2) that $\hat{\beta}\in \sum_{N}^{+}(A)$.
Denote
$$V_{\hat{\beta}}:=\bigcap\limits_{n\geq 0}V_{\hat{\beta}}^{n,n_0}.$$
By assumption (ii) and the continuity of $f_n$, it can be verified that $V_{\hat{\beta}}^{n,n_0}$ are nonempty compact subsets of $V_{b_{0}}$
and satisfy $V_{\hat{\beta}}^{n,n_0}\supset V_{\hat{\beta}}^{n+1,n_0}$ for all $n\geq n_0$.
Hence, $V_{\hat{\beta}}\neq\emptyset$. For any $\beta\in \Omega$, fix one point $x_{\hat{\beta}}\in V_{\hat{\beta}}$. Set
$$S:=\{x_{\hat{\beta}}: \beta\in \Omega\}.$$
It is evident that $\hat{\beta_{1}}\neq \hat{\beta_{2}}$ if and only if $\beta_{1}\neq \beta_{2}$.
So $x_{\hat{\beta_{1}}}\neq x_{\hat{\beta_{2}}}$ for any $\beta_{1},\beta_{2}\in\Omega$ with $\beta_{1}\neq \beta_{2}$. This shows that $S$ is an uncountable set.

{\bf Step 3.} Show that $S$ is a $\delta$-scrambled set of system (2.1), where $\delta$:= min$\{d(V_{i},V_{j}),1\leq i \neq j \leq N\}>0$.

For any given $x_{\hat{\beta_{1}}},x_{\hat{\beta_{2}}} \in S$ with $x_{\hat{\beta_{1}}}\neq x_{\hat{\beta_{2}}}$, the corresponding $\beta_{1}$ and $\beta_{2}$ in $\Omega$ are not equal by (4.3). Then it follows from (4.3) that there exist two different symbols $s_{1},s_{2}\in S_{0}=\{1,\cdots,N\}$ and an infinite increasing sequence
$\{m_{j}\}_{j=1}^{\infty}$ such that $f_{n_0}^{m_{j}}(x_{\hat{\beta_{1}}})\in V_{s_{1}}$ and $f_{n_0}^{m_{j}}(x_{\hat{\beta_{2}}})\in V_{s_{2}}$,
which implies that
$$d(f_{n_0}^{m_{j}}(x_{\hat{\beta_{1}}}),f_{n_0}^{m_{j}}(x_{\hat{\beta_{2}}}))\geq d(V_{s_{1}},V_{s_{2}})\geq \delta,\;\;j\geq 1.$$
So we get that
$$\limsup_{n\rightarrow\infty}d(f_{n_0}^{n}(x_{\hat{\beta_{1}}}),f_{n_0}^{n}(x_{\hat{\beta_{2}}}))\geq \delta.\ \ \ \ \ \eqno(4.4)$$
On the other hand, again by (4.3) one has that
$$\sigma^{k_{n}}(\hat{\beta})=(\underbrace{j_{0}, j_{0},\cdots, j_{0}}_{n},w_{0}, B_{1},\cdots, B_{n-1},\cdots ),$$
where
$$k_{n}=n(n-1)/2+(n-1)l_{1}+(n-2)(n-1)l_{2}/2.$$
So, $f_{n_0}^{k_{n}}(x_{\hat{\beta_{1}}}),f_{n_0}^{k_{n}}(x_{\hat{\beta_{2}}})\in V_{\alpha}^{n-1,k_{n}+n_0}$. Then
$$d(f_{n_0}^{k_{n}}(x_{\hat{\beta_{1}}}),f_{n_0}^{k_{n}}(x_{\hat{\beta_{2}}}))\leq d( V_{\alpha}^{n-1,k_{n}+n_0}).$$
Since $d(V_{\alpha}^{n,m+n_0})$ uniformly converges to 0 with respect to $m\geq0$ as $n\rightarrow\infty$, it yields that
$$\liminf_{n\rightarrow\infty}d(f_{n_0}^{n}(x_{\hat{\beta_{1}}}),f_{n_0}^{n}(x_{\hat{\beta_{2}}}))=0.\ \ \ \ \ \eqno(4.5)$$
Hence, it follows from (4.4) and (4.5) that $S$ is an uncountable $\delta$-scrambled set of system (2.1).

Moreover, for any $x_{0}\in S$, its orbit $\{x_{n}\}_{n=0}^{\infty}$ under system (2.1) lies in $\bigcup\limits_{i=1}^{N}V_{i}$. Then $\{x_{n}\}_{n=0}^{\infty}$ is bounded since $\bigcup\limits_{i=1}^{N}V_{i}$ is bounded.
So, all the orbits starting from the points in $S$ are bounded. Therefore, system (2.1) is chaotic in the strong sense of Li-Yorke.

Consequently, system (1.1) is chaotic in the strong sense of Li-Yorke by assumption (i) and Lemma 2.3.
This completes the proof.
\end{proof}

\section{An Example}

In this section, an example is discussed and shown to be chaotic in the strong sense of Li-Yorke by Theorem 4.2.

\begin{example} Consider the following time-varying logistic system:
$$x_{n+1}=r_{n}x_{n}(1-x_{n}),\;\; n\geq0,\ \ \ \             \eqno(5.1)$$
governed by the maps $f_{n}(x)=r_{n}x(1-x)$, where $r_{n}\geq 9/2$, $n\geq0$.
It is evident that $f_{n}\;(n\geq0)$ are continuous in $(-\infty, +\infty)$. Set
$$V_{1}=[3/5, 1],\; \; V_{2}=[0, 1/3].\ \ \ \ \eqno(5.2)$$
It is clear that $V_{1}$ and $V_{2}$ are disjoint compact subsets of $[0,1]$, and
$$f_{n}(V_{1})=[0, 6r_{n}/25],\; \;  f_{n}(V_{2})=[0, 2r_{n}/9].\ \ \ \ \eqno(5.3)$$
Since $6r_{n}/25>1$ and $2r_{n}/9\geq1$ by $r_{n}\geq 9/2$, by (5.2) and (5.3) one has that
$$f_{n}(V_{1})\cap f_{n}(V_{2})\supset V_{1}\cup V_{2}, \;\;n\geq0.$$
On the other hand, one has that
$$|f_{n}'(x)|=|r_{n}(1-2x)|\geq r_{n}/3\geq3/2>1,\;\; x\in V_{2}.$$
So,
$$|f_{n}(x)-f_{n}(y)|\geq3/2|x-y|,\;\; x\in V_{2}.$$
Hence, all the assumptions in Theorem 4.2 hold for system (5.1) with $a_{ij}=1$ for $i,j=1,2$. By Theorem 4.2, system (5.1) is chaotic in the strong sense of Li-Yorke.
\end{example}

\begin{remark} Note that each of all the results obtained in Sections 3 and 4 can be applied to this example. Comparing to Example 5.1 in [26], we remove the control of upper bounds of $\{r_{n}\}_{n=0}^{\infty}$. So system (5.1) doesn't satisfy assumption (iii) of Lemma 2.8, and thus Lemma 2.8 can not be applied to this example in the case that $\{r_{n}\}_{n=0}^{\infty}$ is not bounded.
\end{remark}


\medskip

Received xxxx 20xx; revised xxxx 20xx.
\medskip

\end{document}